\documentclass[12pt]{article}

\usepackage{amsmath}
\usepackage{amsthm}
\usepackage{amsfonts}
\usepackage{verbatim}
\usepackage{amssymb}
\usepackage{fancyhdr}
\usepackage{graphicx}
\usepackage{float}
\usepackage{pifont}
\usepackage{extarrows}
\usepackage{mathabx}
\usepackage{changepage}
\usepackage{adjustbox}
\usepackage{mdframed}
 \usepackage{graphicx,rotating,booktabs}
\usepackage[verbose]{placeins}
\usepackage{enumitem}
\usepackage{titlesec}
\usepackage{ltablex}
\titleformat*{\section}{\center\Large\bfseries}
\titlespacing{\section}{0pt}{0pt}{\parskip}
\usepackage{caption}

\usepackage[english]{babel}
\usepackage[margin=1in, headsep=4pt]{geometry} 

\usepackage{mathrsfs}
\usepackage{longtable}
\usepackage{tabularx}
    \usepackage{ltablex}

\usepackage{chngcntr}
\allowdisplaybreaks

\newtheorem{bigtheorem}{Theorem}
\newtheorem{theorem}{Theorem}[section]
\newtheorem{lemma}[theorem]{Lemma}
\newtheorem{proposition}[theorem]{Proposition}
\newtheorem{corollary}[theorem]{Corollary}

\newenvironment{remark}[1][Remark]{\begin{trivlist}
\item[\hskip \labelsep {\bfseries #1}]}{\end{trivlist}}

\begin{document}

\title{\bf Generation of finite simple groups by an involution and an element of prime order}
\author {Carlisle S. H. King\\
\it{Imperial College} \\
\it{London}\\
\it{United Kingdom}\\
\it{SW7 2AZ}\\
carlisle.king10@ic.ac.uk}
\date{}
\maketitle

\vspace{-1.3cm}

\section*{\center{Abstract}}
We prove that every non-abelian finite simple group is generated by an involution and an element of prime order. 

\vspace{-1.3cm}

\section*{\center{1. Introduction}}
\stepcounter{section}

\par Given a finite simple group $G$, it is natural to ask which elements generate $G$. Results of Miller \cite{Mi1}, Steinberg \cite{St}, Aschbacher and Guralnick \cite{AsGu} prove that every finite simple group is generated by a pair of elements. A natural refinement is then to ask whether the orders of the generating elements may be restricted: given a finite simple group $G$ and a pair of positive integers $(a,b)$, does there exist a pair of elements $x, y \in G$ with $x$ of order $a$ and $y$ of order $b$ such that $G=\langle x, y \rangle$? If such a pair exists, we say $G$ is $(a,b)$-generated. 
\par As two involutions generate a dihedral group, the smallest pair of interest is $(2,3)$. The question of which finite simple groups are $(2,3)$-generated has been studied extensively. All alternating groups $A_n$ except for $n=3,6,7,8$ are $(2,3)$-generated by \cite{Mi1}. All but finitely many simple classical groups not equal to $PSp_4(2^a), PSp_4(3^a)$ are $(2,3)$-generated by \cite{LiSh}. In fact, recent work by Pellegrini \cite{Pe2} completes the classification of the $(2,3)$-generated finite simple projective special linear groups, which shows that $PSL_n(q)$ is $(2,3)$-generated for $(n,q) \ne (2,9), (3,4), (4,2)$. There is also literature on the $(2,3)$-generation of many other simple classical groups $Cl_n(q)$, showing a positive result for large $n$ explicitly listed (for example, see \cite{SaTa}). All simple exceptional groups except for ${}^2B_2(2^{2m+1})$ (which contain no elements of order 3) are $(2,3)$-generated by \cite{LuMa}. And all sporadic simple groups except for $M_{11}, M_{22}, M_{23}$ and $McL$ are $(2,3)$-generated by \cite{Wo}.
\par Nevertheless, the problem of determining exactly which finite simple groups are $(2,3)$-generated, or more generally $(2,p)$-generated for some prime $p$, remains open. In this paper, we prove:

\begin{bigtheorem}
Every non-abelian finite simple group $G$ is generated by an involution and an element of prime order.
\end{bigtheorem}

 \par By \cite{Mi1}, for $n \ge 5$ and $n \ne 6, 7, 8$, the alternating groups $A_n$ are $(2,3)$-generated, and by \cite{Mi2} these exceptions are $(2,5)$-generated. By \cite{LuMa} the exceptional groups not equal to ${}^2B_2(2^{2m+1})$ are $(2,3)$-generated, and by \cite{Ev} the Suzuki groups are $(2,5)$-generated. By \cite{Wo} the sporadic groups not equal to $M_{11}, M_{22}, M_{23}, McL$, are $(2,3)$-generated, and by \cite{Wi} these exceptions are $(2,p)$-generated for $p=11, 5, 23, 5$ respectively (in fact, all of these exceptions are $(2,5)$-generated, which can be seen using GAP). By Lemma 2.4 below, the 4-dimensional symplectic groups $PSp_4(2^a) \ (a>1), PSp_4(3^a)$ are $(2,5)$-generated, and when combined with Lemmas 2.1 and 2.2, this shows that all finite simple classical groups with natural module of dimension $n \le 7$ (and $P\Omega_8^+(2)$) are $(2,p)$-generated for some $p \in \{3,5,7 \}$.

\par By Zsigmondy's theorem \cite{Z}, for $ q, e > 1$ with $(q,e) \ne (2^a-1, 2), (2, 6),$ there exists a prime divisor $r=r_{q,e}$ of $q^e-1$ such that $r$ does not divide $q^i -1$ for $i<e$. We call $r$ a primitive prime divisor of $q^e-1$. Notice that, in general, $r_{q,e}$ is not uniquely determined by $(q,e)$. In the group $(\mathbb{F}_r)^\times$, $q$ has order $e$, and so $ r \equiv 1 \mod e$. In view of the above discussion, Theorem 1 follows from the following result.

\begin{bigtheorem}
Let $G$ be a finite simple classical group with natural module of dimension $n$ over $\mathbb{F}_{q^\delta}$, where $\delta=2$ if $G$ is unitary and $\delta=1$ otherwise. Assume $n \ge 8$ and $G \ne P\Omega_8^+(2)$. Let $r$ be a primitive prime divisor of $q^e-1$, where $e$ is listed in Table 1. Then $G$ is $(2,r)$-generated.
\end{bigtheorem}

\begin{table}[h]
\begin{center}
\def\arraystretch{1.55}
\caption {Values of $e$ in Theorem 2}
\begin{tabular}{ c  c  }
$G$  & $e$ \\ \hline
$PSL_n(q), PSp_n(q), P\Omega^-_n(q)$ &  $n$  \\ \hline 
$P\Omega^+_n(q)$ & $n-2$ \\ \hline
$P\Omega_n(q) \ (nq \ \text{odd})$ & $n-1$ \\ \hline
$PSU_n(q) \ (n \ \text{odd})$ & $2n$ \\ \hline
$PSU_n(q) \ (n \ \text{even})$ & $2n-2$ \\ \hline
\end{tabular}
\end{center}
\end{table}

\par We note that $r$ is well-defined for the groups described in Theorem 2.
\par There is a large literature on other aspects of the generation of finite simple groups, and we note just a few results. In \cite{MSW}, it is shown that every non-abelian finite simple group other than $PSU_3(3)$ is generated by three involutions. In \cite{GuKa}, and proved independently in \cite{Ste}, it is shown that, given a finite simple group $G$, there exists a conjugacy class $C$ of $G$ such that, given an arbitrary non-identity element $x$ in $G$, there exists an element $y$ in $C$ such that $G= \langle x,y \rangle$. 

\par We now sketch our approach to proving Theorem 2. Let $G$ be any finite group. Let $M <_{\operatorname{max}} G$ denote a maximal subgroup. For a group $H$ let $i_m(H)$ denote the number of elements of order $m$ in $H$. Let $P_{2,p}(G)$ denote the probability that $G$ is generated by a random involution and a random element of order $p$, and let $Q_{2,p}(G) = 1-P_{2,p}(G)$. Then we have
\begin{align*}\tag{$1.1$}
Q_{2,p}(G) &\le \sum_{M <_{\text{max}} G} \frac{i_2(M)}{i_2(G)} \frac{i_p(M)}{i_p(G)},  \\
\end{align*}

\vspace{-0.8cm}

\noindent since the right-hand side is an upper bound for the probability that a random involution and a random element of order $p$ lie in some maximal subgroup of $G$. To prove $G$ is $(2,p)$-generated, it suffices to prove $Q_{2,p}(G)<1$. 
\par In fact, for the proof of Theorem 2 we also need a refinement of (1.1). Let $G, r$ be as in Theorem 2, and let $x \in G$ be an element of order $r$. Let $P_2(G,x)$ denote the probability that $G$ is generated by $x$ and a random involution, and let $Q_2(G,x) = 1-P_2(G,x)$. Then by similar reasoning we have
\begin{align*}\tag{$1.2$}
Q_{2}(G,x) &\le \sum_{x \in M <_{\text{max}} G} \frac{i_2(M)}{i_2(G)}. \\
\end{align*}

\vspace{-0.8cm}

\noindent  To prove $G$ is $(2,r)$-generated, it suffices to prove $Q_{2}(G,x)<1$. Our method in most cases is to determine the maximal subgroups $M$ of $G$ containing $x$, and then bound $i_2(M)$ and $i_2(G)$ in terms of $n$ and $q$ such that for $n$ and $q$ sufficiently large we have $Q_2(G,x)<1$. For the remaining cases with small $n$ and $q$ we improve the bounds case by case using literature such as \cite{BHRD}. 

\paragraph{\it Acknowledgements.} This paper is part of work towards a PhD degree under the supervision of Martin Liebeck, and the author would like to thank him for his guidance throughout. The author is also grateful for the financial support from EPSRC.

\vspace{-1cm}

\section*{\center{2. Preliminary results}}
\stepcounter{section}

The $(2,3)$-generation of classical groups has been studied extensively, and there are many results for groups of small dimension that we will make use of. 

\begin{lemma}
If $G$ is a finite simple classical group listed in Table 2, then $G$ is $(2,3)$-generated.
\end{lemma}

\begin{table}[h]
\begin{center}
\def\arraystretch{1.55}
\caption {Some $(2,3)$-generated finite simple classical groups}
\begin{tabular}{ c  c  c  }
$G$  & Exceptions & References \\ \hline
$PSL_n(q), 2 \le n \le 7$ & $(n,q) = (2,9), (3, 4), (4,2)$ & \cite{Mac, PeTa, PeTaVs, TabTc, Tab}  \\ \hline 
$PSp_n(q),  4 \le n \le 6$ & $n=4, q = 2^a, 3^a$ & \cite{PeTaVs, Pe}  \\ \hline
$P\Omega_n(q), n =7 $ & &  \cite{Pe}  \\ \hline
$PSU_n(q), 3 \le n \le 7$ & $(n,q) = (3,3), (3,5), (4,2), (4,3), (5,2)$ & \cite{PeTa, PeTaVs, PePrTa, Pe} \\ \hline
\end{tabular}
\end{center}
\end{table}
\noindent We note that though there are many other $(2,3)$-generation results regarding classical groups, our method will not require them.

\begin{lemma}
If $G$ is listed in Table 3 then $G$ is $(2,p)$-generated, where $p \in \{ 5,7 \}$ is specified.
\end{lemma}

\begin{table}[h]
\begin{center}
\def\arraystretch{1.3}
\caption {Cases when $G$ is small and $(2,p)$-generated}
\begin{tabular}{ c  c  c}
$G$ & $(n,q)$ & $p$\\ \hline
$PSL_n(q)$ & $(2,9), (4,2)$ & $5$\\ 
& $(3, 4)$ & $7$ \\ \hline
$PSU_n(q)$ &  $(4,2), (5,2)$ & $5$ \\ 
& $  (3,3), (3,5), (4,3)$ & $7$ \\ \hline
$P\Omega^+_8(2)$ & & $5$ \\ \hline

\end{tabular}
\end{center}
\end{table}
\vspace{-0.3cm}

\begin{proof}
\par Let $G=PSL_3(4)$ and $p=7$. By \cite{Atlas} the only maximal subgroups $M$ of $G$ with order divisible by 7 are isomorphic to $PSL_2(7)$. The index of $M=PSL_2(7)$ in $G$ is 120 and there are 3 $G$-conjugacy classes of such subgroups. We find $i_2(PSL_2(7))=21, i_2(G)=315, i_7(PSL_2(7))=48, i_7(G)=5760$. Therefore by $( 1.1 )$,
\begin{align*}
Q_{2,7}(G) &\le \sum_{M <_{\operatorname{max}}G} \frac{i_2(M)}{i_2(G)} \frac{i_7(M)}{i_7(G)} \\
&= 360 \frac{i_2(PSL_2(7))}{i_2(G)} \frac{i_7(PSL_2(7))}{i_7(G)} \\
&= \frac{1}{5},
\end{align*}
\noindent and so $G$ is $(2,7)$-generated. The results for the remaining groups are proved similarly. \end{proof}

\begin{remark} We note that the groups in Table 3 are not $(2,3)$-generated: it is elementary to prove that $PSL_2(9)$ is not $(2,3)$-generated; for $PSL_3(4), PSU_3(3), PSU_3(5),$ this is proved in \cite{PeTa}; for $PSL_4(2), PSU_4(3)$, this is proved in \cite{PeTaVs}; for $PSU_4(2) \cong PSp_4(3),$ this is proved in \cite{LiSh}; for $PSU_5(2)$, this is proved in \cite{Vs}; and for $P\Omega_8^+(2)$, this is proved by Vsemirnov. We also note that $PSL_3(4), PSU_3(5)$ and $PSU_4(3)$ are actually $(2,5)$-generated; this can be seen using GAP. However, the method discussed using (1.1) fails in these cases, and so we do not prove this statement.
\end{remark}

\par As discussed in the preamble to Theorem 2, in order to reduce our study of $(2,p)$-generation of finite simple groups to the groups covered by Theorem 2 we need to prove that the 4-dimensional symplectic groups $PSp_4(2^a) \ (a>1), PSp_4(3^a)$ are $(2,p)$-generated for some prime $p$. We first state \cite[Proposition~1.3]{LaLiSe}, which will be useful for later results.

\begin{proposition} 
Let $Y$ be a simple algebraic group over $K$, an algebraically closed field of characteristic $p>0$, and let $N$ be the number of positive roots in the root system of $Y$. Suppose that $F$ is a Frobenius morphism of $Y$ such that $G=(Y^F)'$ is a finite simple group of Lie type over $\mathbb{F}_q$. Assume $G$ is not of type ${}^2F_4, {}^2G_2$ or ${}^2B_2$, and define $N_2=\operatorname{dim}Y-N$. Then
\begin{equation*}i_2(\operatorname{Aut}G) < 2(q^{N_2}+q^{N_2-1}).
\end{equation*}
\end{proposition}

\vspace{0.1cm}
\begin{lemma}
Let $G=PSp_4(q)$ where $q=2^a$ or $3^a$, $q \ne 2$. Then $G$ is $(2,5)$-generated.
\end{lemma}

\begin{remark} We note that by \cite[Theorem~1.2]{PeTaVs}, $PSp_4(3^{2a})$ is $(2,5)$-generated. However, we do not use this result; we prove Lemma 2.4 by our usual method using $(1.1)$.
\end{remark}

\begin{proof}
\par Suppose $q=2^a$. Details on the conjugacy classes of $G$ can be found in \cite{E}. We find $i_2(G)=(q^2+1)(q^4-1), i_5(G) \ge q^3(q-1)(q^2+1)(q^2-q+4)$. By \cite{BHRD} the possible maximal subgroups with order divisible by 5 are isomorphic to $[q^3]:GL_2(q), Sp_2(q) \wr S_2, Sp_2(q^2).2, SO_4^\pm(q), Sz(q)$ or $Sp_4(q^\frac{1}{t})$ for some prime $t$. 
\par If $M \cong [q^3]:GL_2(q)$, then by \cite{E} we compute $i_2(M)=(q-1)(q^3+2q^2+q+1), i_5(M) \le 2q^3(q+1)(2q+5)$. By \cite{BHRD} there are two $G$-conjugacy classes of subgroups of this type, and $|G:M|=(q+1)(q^2+1)$.
\par If $M \cong Sp_2(q) \wr S_2$ then $i_2(M)=(i_2(Sp_2(q))+1)^2-1+|Sp_2(q)| = q^4+q^3-q-1, \\ i_5(M) = (i_5(Sp_2(q))+1)^2-1 \le 4q(q^3+2q^2+2q+1)$. There is a single $G$-conjugacy class of subgroups of this type, and $|G:M|=\frac{q^2(q^2+1)}{2}$.
\par If $M \cong Sp_2(q^2).2$, then by Proposition 2.3 we have $i_2(M) < 2q^2(q^2+1)$, and we also have $i_5(M) =i_5(Sp_2(q^2)) \le 2q^2(q^2+1)$. There is a single $G$-conjugacy class of subgroups of this type, and $|G:M|=\frac{q^2(q^2-1)}{2}$.
\par Suppose $M \cong Sp_4(q^\frac{1}{t})$ for some prime $t$. Then by \cite{E} we have $i_2(M) = (q^\frac{2}{t}+1)(q^\frac{4}{t}-1), i_5(M) \le q^\frac{3}{t}(q^\frac{1}{t}+1)(q^\frac{2}{t}+1)(q^\frac{2}{t}+q^\frac{1}{t}+4)$. There are fewer than $\log_2(q)$ $G$-conjugacy classes of such subgroups, each with $|G:M| = \frac{q^4(q^2-1)(q^4-1)}{q^\frac{4}{t}(q^\frac{2}{t}-1)(q^\frac{4}{t}-1)}$. 
\par If $M \cong SO^+_4(q) \cong SL_2(q) \wr S_2$ then $i_2(M) = (i_2(SL_2(q))+1)^2-1+|SL_2(q)| = q^4+q^3-q-1, i_5(M)=(i_5(SL_2(q))+1)^2-1 \le 4q(q^3+2q^2+2q+1)$ as above. There is a single $G$-conjugacy class of subgroups of this type, and $|G:M|=\frac{q^2(q^2+1)}{2}$.
\par If $M \cong SO^-_4(q) \cong SL_2(q^2).2$, then by Proposition 2.3 we have $i_2(M) < 2q^2(q^2+1)$ and $i_5(M) = i_5(SL_2(q^2)) \le 2q^2(q^2+1)$ as above. There is a single $G$-conjugacy class of subgroups of this type, and $|G:M|=\frac{q^2(q^2-1)}{2}$. 
\par Finally suppose $M=Sz(q)$. Then by \cite{Suz} we have $i_2(M)=(q-1)(q^2+1), i_5(M) \le q^2(q+\sqrt{2q}+1)(q-1)$. There is a single $G$-conjugacy class of subgroups of this type, and $|G:M|=q^2(q+1)(q^2-1)$.
\par Therefore by $(1.1)$, 
\begin{align*}
Q_{2,5}(G) &\le \sum_{M <_{\operatorname{max}}G} \frac{i_2(M)}{i_2(G)} \frac{i_5(M)}{i_5(G)} \\[2pt]
&=  \Bigg( \sum_{M \cong [q^3]:GL_2(q)} + \sum_{M \cong Sp_2(q) \wr S_2}+\sum_{M \cong Sp_2(q^2).2} +\sum_{M \cong Sp_4(q^\frac{1}{t})} \\[0pt]
& \hspace{6cm} + \sum_{M \cong SO_4^+(q)} +\sum_{M \cong SO_4^-(q)} +\sum_{M \cong Sz(q)} \Bigg)   \frac{i_2(M)}{i_2(G)} \frac{i_5(M)}{i_5(G)} \\[1pt]
& \le \Big(4q^3(q-1)(q+1)^2(2q+5)(q^2+1)(q^3+2q^2+q+1) \\[-2pt]
&\hspace{2cm} + 2q^3(q^2+1)(q^3+2q^2+2q+1)(q^4+q^3-q-1) \\[3pt]
&\hspace{3.5cm} + 2q^6(q^2-1)(q^2+1)^2 \\[3pt]
&\hspace{4.8cm} + 2\log_2(q)q^\frac{7}{2}(q^\frac{1}{2}+1)(q+1)(q+q^\frac{1}{2}+4)(q^2-1)(q^4-1) \\[3pt]
&\hspace{3.5cm} + 2q^3(q^2+1)(q^3+2q^2+2q+1)(q^4+q^3-q-1) \\[3pt]
&\hspace{2cm} + 2q^6(q^2-1)(q^2+1)^2 \\[-2pt]
&\hspace{-1.25cm} + q^4(q-1)^2(q+1)(q+\sqrt{2q}+1)(q^2-1)(q^2+1) \Big) \frac{1}{q^3(q-1)(q^2+1)^2(q^2-q+4)(q^4-1)}, 
\end{align*}
\noindent and it is straightforward to verify $Q_{2,5}<1$ for $q \ge 8$. If $q=4$ then a similar argument using \cite{Atlas} yields the result. This completes the proof for $q=2^a$.
\par The argument for $q=3^a$ is similar (using \cite{ShMo} instead of \cite{E} for details on the conjugacy classes of $G$). 
\end{proof}

\vspace{-1.7cm}

\section*{\center{3. Involutions in classical groups}}
\stepcounter{section}
\par Let $G$ be a finite simple classical group with natural module of dimension $n$ defined over the field $\mathbb{F}_{q^\delta}$, where $\delta=2$ if $G$ is unitary and $\delta=1$ otherwise. In this section we find a lower bound for $i_2(G)$, the number of involutions in $G$.

\begin{proposition}
The number $i_2(G)$ of involutions satisfies $i_2(G) \ge I_2(G)$, for $I_2(G)$ is given in Table 4.
\end{proposition}

\begin{proof}
\par  For each $G$ we specify an involution $y$ and calculate $|y^G|$ to give a lower bound for $i_2(G)$. The choices for $y$ and values $|y^G|$ are listed in Table 5. It is then elementary to obtain the bounds stated in Table 4. We let $[d]$ denote an arbitrary group of order $d$. The notation of the elements in the table is as follows:
\par  If $q$ is odd define the following (projective) involutions in $GL_n(q)$:
\begin{align*}
s &= 
\begin{pmatrix}
i I_{\frac{n}{2}} \\
& -i I_{\frac{n}{2}}
\end{pmatrix} \hspace{0.15cm} \text{if $n$ is even, $q \equiv 1 \ \text{mod} \ 4$ and $i^2=-1$,} \\
t &=
\begin{pmatrix}
& I_\frac{n}{2} \\
-I_{\frac{n}{2}}
\end{pmatrix} \hspace{0.48cm} \text{if $n$ is even,} \\
u_k &=  \begin{pmatrix}
 I_k \\
& -I_{n-k}
\end{pmatrix}, \hspace{0.1cm} \text{$0 <k<n$;}
\end{align*}
 
\par If $q$ is even, define the following involutions in $GL_n(q)$:
\begin{align*}
j_k = \begin{pmatrix} I_k \\ & I_{n-2k} \\ I_k & & I_k \end{pmatrix},  \hspace{0.1cm}  0<k \le \frac{n}{2}.
\end{align*}

\par In the case where $q$ is odd, conjugacy class representatives of involutions are found in Table 5.5.1 of \cite{GoLySo}, listed alongside information on their respective centralizers. 
\par In the case where $q$ is even and $G=PSL^\epsilon_n(q)$, involutions are conjugate to $j_k$ for some $k$. The structures of the centralizers of $j_k$ are found in Sections 4 and 6 of \cite{AsSe}, and exact conjugacy class sizes can be calculated using \cite[Theorem~7.1]{LiSe}. If $q$ is even and $G$ is symplectic or orthogonal, involutions are of the form $a_k, b_k$ or $c_k$ with Jordan normal form $j_k$ by \cite{AsSe}. In Sections 7 and 8 of the same paper the structure of the centralizer of each involution is given, and exact conjugacy class sizes can be calculated using \cite[Theorem~7.2]{LiSe}.\end{proof}

\begin{table}[h]
\begin{center}
\def\arraystretch{1.7}
\caption {Lower bounds for $i_2(G)$}
\begin{tabular}{ c  c  }
$G$  & $I_2(G) $  \\ \hline
$PSL^\epsilon_n(q)$ &  $\frac{1}{8}q^{\lfloor \frac{n^2}{2} \rfloor}$ \\
$PSp_n(q)$ & $\frac{1}{2}q^{\frac{n^2}{4}+\frac{n}{2}}$\\
$P\Omega^\epsilon_n(q)$ &  $\frac{1}{8}q^{\frac{n^2}{4}-1}$ \\
$P\Omega_n(q), nq$ odd & $ \frac{1}{2}q^{\frac{n^2-1}{4}}$ \\  \hline
\end{tabular}
\end{center}
\end{table}

\vspace{0.0cm}

\renewcommand{\arraystretch}{1.55}

\setlength\LTleft{-1.5cm}
\footnotesize
\begin{longtable}{c  c  c  c}
\caption{Values for $|y^G|$} \\
$G$  & Conditions & $y$ & $|y^G|$  \\ \hline
$PSL^\epsilon_n(q), n \ge 2$ & $n$ even, $q$ even & $j_\frac{n}{2} $ & $|GL^\epsilon_n(q):[q^{\frac{n^2}{4}}].GL^\epsilon_\frac{n}{2}(q)|$ \\
& $n$ even, $ q \equiv \epsilon \mod 4$ & $s$ &  $|GL^\epsilon_n(q):GL^\epsilon_\frac{n}{2}(q)^2.2|$ \\
& $n$ even, $q \equiv -\epsilon \mod 4$ & $t$ & $|GL^\epsilon_n(q):GL_{\frac{n}{2}}(q^2).2|$ \\
& $n$ odd, $q$ even & $j_\frac{n-1}{2}$ & $|GL^\epsilon_n(q):[q^{\frac{n^2+2n-7}{4}}].\left(GL^\epsilon_\frac{n-1}{2}(q)\times GL^\epsilon_1(q) \right)|$ \\
& $n$ odd, $q$ odd & $(-1)^{\frac{n+1}{2}} u_\frac{n-1}{2}$ & $|GL^\epsilon_n(q):GL^\epsilon_\frac{n-1}{2}(q) \times GL^\epsilon_\frac{n+1}{2}(q)|$\\[0.3em] \hline \noalign{\vskip0.1em}
$PSp_n(q), n \ge 4$ & $\frac{n}{2}$ even, $q$ even & $c_{\frac{n}{2}}$ & $ |Sp_n(q): [q^{ \frac{1}{2}(\frac{n^2}{4}+\frac{3n}{2}-2)}].Sp_{\frac{n}{2}-2}(q)|$ \\
& $\frac{n}{2}$ odd, $q$ even & $b_{\frac{n}{2}}$ & $|Sp_n(q): [q^{\frac{n}{4}(\frac{n}{2}+1)}].Sp_{\frac{n}{2}-1}(q)|$ \\
& $q \equiv 1 \mod 4$ & $s$ & $|Sp_n(q):GL_\frac{n}{2}(q).2|$\\
& $q \equiv 3 \mod 4$ & $t$ & $|Sp_n(q):GU_{\frac{n}{2}}(q).2|$ \\ \hline  \noalign{\vskip0.1em}
$P\Omega^+_n(q), n \ge 8$ & $\frac{n}{2}$ even, $q$ even & $c_\frac{n}{2}$ & $|\Omega^+_n(q):[q^{ \frac{1}{2}(\frac{n^2}{4}+\frac{n}{2}-2)}].Sp_{\frac{n}{2}-2}(q) \cap \Omega_n^+(q)|$ \\
& $\frac{n}{2}$ even, $q^\frac{n}{4} \equiv \epsilon \mod  4$  & $u_\frac{n}{2}$ & $|\Omega^+_n(q): O^\epsilon_\frac{n}{2}(q)^2 .2 \cap \Omega^+_n(q)|$ \\
& $\frac{n}{2}$ odd, $q$ even & $c_{\frac{n}{2}-1}$ & $|\Omega^+_n(q) :[q^{\frac{1}{2}(\frac{n^2}{4}+\frac{3n}{2}-10)}] .\left(  Sp_{\frac{n}{2}-3}(q) \times Sp_2(q) \right)\cap \Omega^+_n(q)| $ \\
& $\frac{n}{2}$ odd, $q$ odd & $u_{\frac{n}{2}-1}$ & $|\Omega^+_n(q):O_{\frac{n}{2}-1}^+(q) \times O_{\frac{n}{2}+1}^+(q) \cap \Omega^+_n(q)|$ \\[0.3em] \hline  \noalign{\vskip0.1em}
$P\Omega_n^-(q), n \ge 8$ & $\frac{n}{2}$ even, $q$ even & $c_\frac{n}{2}$ & $|\Omega^-_n(q):[q^{ \frac{1}{2}(\frac{n^2}{4}+\frac{n}{2}-2)}].Sp_{\frac{n}{2}-2}(q) \cap \Omega_n^-(q)|$ \\
& $\frac{n}{2}$ even, $q$ odd & $u_\frac{n}{2}$ & $|\Omega^-_n(q):\left( O_\frac{n}{2}^+(q) \times O_\frac{n}{2}^-(q) \right).2 \cap \Omega^-_n(q)|$ \\
& $\frac{n}{2}$ odd, $q$ even & $c_{\frac{n}{2}-1}$ & $|\Omega^-_n(q) :[q^{\frac{1}{2}(\frac{n^2}{4}+\frac{3n}{2}-10)}].\left(  Sp_{\frac{n}{2}-3}(q) \times Sp_2(q) \right) \cap \Omega^-_n(q)| $ \\
& $\frac{n}{2}$ odd, $q^{\frac{n-2}{4}} \equiv \epsilon \mod 4$ & $u_{\frac{n}{2}-\epsilon}$ & $|\Omega^-_n(q):O_{\frac{n}{2}-1}^\epsilon(q) \times O_{\frac{n}{2}+1}^{-\epsilon}(q) \cap \Omega^-_n(q)|$\\[0.3em] \hline
$P\Omega_n(q) \ (nq \ \text{odd}), n \ge 7$  & $q^{\lfloor \frac{n+1}{4} \rfloor} \equiv \epsilon \mod 4$ & $(-1)^{\frac{n+1}{2}}u_\frac{n-1}{2}$ & $|\Omega_n(q):O^\epsilon_{2\lfloor \frac{n+1}{4} \rfloor}(q) \times O_{n-2\lfloor \frac{n+1}{4} \rfloor}(q) \cap \Omega_n(q)|$ \\[0.3em] \hline

\end{longtable}
\normalsize
\setlength\LTleft{0cm}

\vspace{-1.39cm}

\section*{\center{4. Maximal subgroups with order divisible by $r$}}
\stepcounter{section}

Let $G$ be a finite simple classical group described in Theorem 2 -- that is, with natural module $n \ge 8$ and $G \ne P\Omega_8^+(2)$. If $G=P\Omega^\epsilon_n(q)$, let $|PSO^\epsilon_n(q):P\Omega^\epsilon_n(q)|=a_\epsilon \in \{ 1,2 \}$ and let $|Z(\Omega^\epsilon_n(q))|=z_\epsilon \in \{1,2\}$ . The values $a_\epsilon, z_\epsilon$ can be found in \cite[\S~2]{KlLi}. Recall the definition of $r$, a primitive prime divisor of $q^e-1$ for $e$ listed in Table 1.

\par The subgroup structure of $G$ is well-understood due to a theorem of Aschbacher \cite{As}. The theorem states that if $M$ is a maximal subgroup of $G$, then $M$ lies in a natural collection $\mathscr{C}_1, \dots, \mathscr{C}_8$, or $M \in \mathscr{S}$. Subgroups in $\mathscr{C}_i$ are described in detail in \cite{KlLi}, where the structure and number of conjugacy classes are given. Subgroups in class $\mathscr{S}$ are almost simple groups which act absolutely irreducibly on the natural module $V$ of $G$. The orders of subgroups in classes $\mathscr{C}_1, \dots, \mathscr{C}_8$ can easily be computed using \cite[\S~4]{KlLi}, and this yields the result below.

\begin{proposition}
Let $M$ be a maximal subgroup of $G$ with order divisible by $r$, and assume $M$ lies in one of the Aschbacher classes $\mathscr{C}_1, \dots, \mathscr{C}_8$. Then $M$ is conjugate to a subgroup listed in Table 6. The number of $G$-conjugacy classes of each type, $c_M$, is also listed. 
\end{proposition}

 \par Table 6 also lists bounds $I_2(M)$ and $ N_M$, where $i_2(M) \le I_2(M)$ and $N_M \le |N_M(\langle x \rangle)|$ where $x \in G$ is an element of order $r$ contained in $M$. These bounds are justified in Propositions 5.1 and 6.4 below.

\begin{sidewaystable}[!]
\renewcommand{\arraystretch}{1.85}
    \caption{Maximal subgroups $M \notin \mathscr{S}$ with $r \mid |M|$}
  \bigskip
    \centering
\footnotesize
\setlength\tabcolsep{2.5pt}
 \hspace*{-1cm}
\begin{tabular}{c  c  c  c  c  c c}
           \toprule
$G$  & Class & Type of $M$ & Conditions & $c_M$ & $I_2(M)$ & $N_M$ \\ \hline
$PSL_n(q), n \ge 2$ & $\mathscr{C}_3$ & $GL_k(q^t).t$ & $n=kt$, $t$ prime & 1 &  $2\frac{q^{2t}-1}{q-1}q^{\frac{n^2}{2t}+\frac{n}{2}-2t}$ & $|N_G(\langle x \rangle)|$  \\
& $\mathscr{C}_6$ & $2^{2k}Sp_{2k}(2)$ & $n=2^k, r=2^k+1, q$ odd & $(q-1, n)$ & $2^{\log_2(n)(2\log_2(n)+3)}$ & $r$ \\
& $\mathscr{C}_8 $ & $PSp_n(q)$ & $n$ even & $(q-1, \frac{n}{2})$  & $2(q+1)q^{\frac{n^2+2n-4}{4}}$ & $ \frac{n(q^\frac{n}{2}+1)}{(2, q-1)}$\\
& & $PSO^-_n(q)$ & $n$ even, $q$ odd & $\frac{(q-1, n)}{2}$ & $2(q+1)q^{\frac{n^2-4}{4}}$ & $\frac{n(q^\frac{n}{2}+1)}{2a_-}$\\
& & $PSU_n(q^\frac{1}{2})$ & $n$ odd, $q$ square & $\frac{q-1}{[q^\frac{1}{2}+1, \frac{q-1}{(q-1, n)}]}$ & $ 2(q^\frac{1}{2}+1)q^{\frac{n^2+n-4}{4}}$ & $\frac{n(q^\frac{n}{2}+1)}{(q^\frac{1}{2}+1)(n, q^\frac{1}{2}+1)}$ \\[0.6em] \hline
$PSp_n(q), n \ge 4$ & $\mathscr{C}_3$ & $Sp_k(q^t).t$ & $n=kt, k$ even, $t$ prime & 1 & $2(q^t+1)q^{\frac{n^2}{4t}+\frac{n}{2}-t}$ & $|N_G(\langle x \rangle)|$\\
& & $GU_\frac{n}{2}(q).2$ & $\frac{n}{2}$ odd, $q$ odd & 1 & $(q+1)^2q^\frac{n^2+2n-16}{8}$ & $|N_G(\langle x \rangle)|$ \\
& $\mathscr{C}_6$ & $2^{2k}O_{2k}^-(2)$ & $n=2^k, q=p$ odd, $r=n+1$ & 1 or 2 & $2^{2\log^2_2(n)+\log_2(n)+1}$ & $r$ \\
& $\mathscr{C}_8$ & $PSO_n^-(q)$ & $q$ even & 1 & $2(q+1)q^{\frac{n^2-4}{4}}$ & $ \frac{n(q^\frac{n}{2}+1)}{2}$  \\ \hline
$P\Omega^+_n(q), n \ge 8$ & $\mathscr{C}_1$ & $O^-_{n-2}(q) \times O^-_2(q)$ & & 1 & $2(q+1)^2q^{\frac{n^2-4n}{4}}$ & $|N_G(\langle x \rangle)|$  \\
& & $O_{n-1}(q) \times  O_{1}(q)$ & $q$ odd & 2 & $\frac{4}{z_+}(q+1)q^{\frac{n^2-2n-4}{4}}$ & $\frac{(n-2)(q^{\frac{n}{2}-1}+1)}{a_+}$ \\
& & $O_{n-1}(q)$ & $q$ even & 1 & $2(q+1)q^{\frac{n^2-2n-4}{4}}$ & $\frac{(n-2)(q^{\frac{n}{2}-1}+1)}{2}$ \\
& $\mathscr{C}_2$ & $O_1(q) \operatorname{wr} S_n$ & $q=p$ odd, $r=n-1$ & 2 or 4  & $2^{n-1}n! $ & $r$ \\
& $\mathscr{C}_3$ & $O_\frac{n}{2}(q^2).2$ & $\frac{n}{2}$ odd, $q$ odd & 1 or 2 & $\frac{4}{z_+}(q+1)q^{\frac{n^2-20}{8}}$ & $\frac{(n-2)(q^{\frac{n}{2}-1}+1)}{4a_+}$ \\
& & $GU_\frac{n}{2}(q).2$ & $\frac{n}{2}$ even & 2 & $ \frac{2}{z_+}(q+1)^2q^{\frac{n^2+2n-16}{8}}$ & $|N_G(\langle x \rangle)|$ \\
& $\mathscr{C}_6$ & $2^{2k}O^+_{2k}(2)$ & $n=2^k, q=p$ odd, $r=n-1$ & 4 or 8 & $2^{\log_2(n)(2\log_2(n)+1)}$ & $ r$\\ \hline
$P\Omega_n^-(q), n \ge 8$ & $\mathscr{C}_3$ & $O_k^-(q^t).t$ & $n=kt, t$ prime, $k \ge 4$ & 1 & $2(q^t+1)q^{\frac{n^2}{4t}-t}$ & $|N_G(\langle x \rangle)|$\\
& & $GU_{\frac{n}{2}}(q).2$ & $\frac{n}{2}$ odd & 1 & $\frac{2}{z_-}(q+1)^2q^\frac{n^2+2n-16}{8}$  & $|N_G(\langle x \rangle)|$\\ \hline
$P\Omega_n(q), n \ge 7, nq$ odd & $\mathscr{C}_1$ & $O^-_{n-1}(q) \times O_1(q)$ & & 1 & $4(q+1)q^{\frac{n^2-2n-3}{4}}$  & $|N_G(\langle x \rangle)|$\\
& $\mathscr{C}_2$ & $O_1(q) \operatorname{wr} S_n$ & $q=p, r=n$ & 1 or 2 & $2^{n-1}n! $ & $r$ \\ \hline
$PSU_n(q), n \ge 3, n$ odd & $\mathscr{C}_3$ & $GU_k(q^t).t$ & $n=kt, t$ prime, $t \ge 3$  & 1 & $ \frac{2(q^t+1)^2}{q+1}q^{\frac{n^2}{2t}+\frac{n}{2}-2t}$ & $|N_G(\langle x \rangle)|$ \\ \hline
$PSU_n(q), n \ge 4, n$ even & $\mathscr{C}_1$ & $GU_{n-1}(q) \times GU_1(q)$ & & 1  & $2(q+1)q^{\frac{n^2-n-4}{2}}$ & $|N_G(\langle x \rangle)|$\\ \hline
\end{tabular}\hspace*{-1cm}
\end{sidewaystable}

\par The subgroups $M \in \mathscr{S}$ with order divisible by $r$ are given by \cite{GPPS}. We first list these subgroups with $\operatorname{soc}(M) \in \operatorname{Lie}(p')$, where $\operatorname{Lie}(p')$ is the set of finite simple groups of Lie type with characteristic not equal to $p$.

\par To state the next result, let 
\begin{align*} 
e_G = \left\{ 
  \begin{array}{l l}
(n, q-\epsilon) \hspace{0.5cm} &\text{if} \ G=PSL^\epsilon_n(q) \\[-0.1cm]
(2,q-1) &\text{if} \ G=PSp_n(q) \\[-0.1cm]
a_\epsilon(2, q-1)^2 &\text{if} \ G=P\Omega^\epsilon_n(q) \ (n \ \text{even})\\[-0.1cm]
2 &\text{if} \ G=P\Omega_n(q) \ (nq \  \text{odd}). \\[-0.1cm]
\end{array}
\right. \\
\end{align*}

\vspace{-0.8cm}

\noindent Then $e_G$ is the index of $G$ in the projective similarity group of the same type (see \cite[\S~2]{KlLi}).

\begin{proposition}
Let $M$ be a maximal subgroup of $G$ with order divisible by $r$, and assume $M \in \mathscr{S}$ with $\operatorname{soc}(M) \in \operatorname{Lie}(p')$. The possibilities for $\operatorname{soc}(M)$ are listed in Table 7. Upper bounds $C_{M}$ are given for the number of $G$-conjugacy classes of maximal subgroups with each listed socle.
\end{proposition}

\begin{proof}

\par Let $M \in \mathscr{S}$ with $\operatorname{soc}(M) \in \operatorname{Lie}(p')$, and let $V$ be the natural module of $G$. The possibilities for $\operatorname{soc}(M)$ are determined  in \cite{GPPS}, and we list them in Table 7 alongside necessary conditions.
\par It remains to justify the bound $C_{M}$. Consider $\operatorname{soc}(M)=S=PSL_d(s), d \ge 3$. Here $\dim V=n=\frac{s^d-1}{s-1}-\delta$ where $\delta \in \{0, 1 \}$. By \cite[Theorem~1.1]{GuTi1}, in most cases the representation of $S$ on $V$ is one of at most $s-1$ Weil representations. The exceptions are when $d=3, s=2$ or $4$, and in each case it can be verified that the number of possible representations is again at most $s-1$  using \cite{Atlas} and \cite{ModAtlas}. Hence the number of $PGL(V)$-classes of such subgroups $M$ is at most $s-1$. If $G=PSL_n(q)$, it follows that that number of $G$-classes of subgroups $M$ is at most $(s-1)|PGL_n(q):PSL_n(q)|=(s-1) e_G$. If $G=P\Omega^\epsilon_n(q)$, Corollary 2.10.4 of \cite{KlLi} shows that the number of $PGO^\epsilon_n(q)$-classes of subgroups $M$ is also at most $s-1$, and so the number of $G$-classes is at most $(s-1)|PGO^\epsilon_n(q):P\Omega^\epsilon_n(q)|=(s-1)e_G$.
\par The remaining possibilities for $\operatorname{soc}(M)$ and $G$ are dealt with similarly, using \cite{GMST}, \cite{HiMal}, \cite{Bu} for $\operatorname{soc}(M)=PSp_{2d}(s), PSU_d(s), PSL_2(s)$ respectively. \end{proof}

\vspace{-0.3cm}

\renewcommand{\arraystretch}{1.5}
\begin{footnotesize}
\setlength\LTleft{0.65cm}
\begin{longtable}{c  c  c  c }

\caption {Maximal subgroups $M \in \mathscr{S}$ with $\operatorname{soc}(M) \in \operatorname{Lie}(p'), r \mid |M|$} \\

$G$  & $\operatorname{soc}(M)$ & Conditions & $C_{M}/e_G$ \\ \hline
$PSL_n(q), PSp_n(q),$ & $PSL_d(s)$ & $d \ge 3$ prime, $n=\frac{s^d-1}{s-1}-1, r=n+1$ & $s-1$   \\
$  P\Omega^-_n(q)$ & $PSp_{2d}(s)$ & $s \ne 3$ odd, $d =2^b \ge 2, n=\frac{1}{2}(s^d-1), r=n+1$ &  $4$    \\
& $PSU_d(s)$ & $d$ prime, $n=\frac{s^d+1}{s+1}-1, r=n+1$ & $s+1$    \\
& $PSL_2(n)$ & $n=2^b, b=2^{b'}, r=n+1$ & $1$ \\
& $PSL_2(n+1)$ & $r=n+1$ & $\frac{n}{4}$ \\
& $PSL_2(2n+1)$ & $r=n+1$ or $2n+1$ & $2$ \\ \hline
$P\Omega^+_n(q)$ & $PSp_{2d}(3)$ & $n=\frac{3^d+1}{2}$ odd, $d \ge 3$ prime, $r=n-1$ &  $4$  \\
& $PSL_2(n-1)$ & $r=n-1$ & $\frac{n-4}{4}$ \\
& $PSL_2(n)$ & $n=2^b, b$ prime, $r=n-1$ & $1$  \\
& $PSL_2(2n-1)$ & $r=n-1$  & $2$  \\  \hline
$P\Omega_n(q), nq$ odd & $PSL_d(s)$ & $d \ge 3$ prime, $n=\frac{s^d-1}{s-1}, r=n$ & $s-1$   \\
 & $PSp_{2d}(s)$ & $s \ne 3$ odd, $d =2^b \ge 2, n=\frac{1}{2}(s^d+1), r=n$ &  $4$   \\
 & $PSp_{2d}(3)$ & $d$ an odd prime, $n=\frac{1}{2}(3^d-1), r=n$ &  $4$   \\
 & $PSU_d(s)$ & $d$ prime, $n=\frac{s^d+1}{s+1}, r=n$ & $s+1$  \\
& $PSL_2(n-1)$ & $r=n$ & $\frac{n-3}{2}$ \\
& $PSL_2(n)$ & $n=2^b, b=2^{b'}, r=n$ & $1$  \\
& $PSL_2(n+1)$ & $r=n$ & $\frac{n+1}{2}$ \\
 & $PSL_2(2n-1)$ & $r=n$ or $2n-1$ & $2$ \\
& $PSL_2(2n+1)$ & $r=n$ & $2$ \\  \hline
$PSU_n(q), n$ odd & $PSL_2(2n+1)$ & $r=2n+1$  & $2$ \\ \hline
$PSU_n(q), n$ even & $PSL_2(2n-1)$ & $r=2n-1$ & $2$  \\ \hline
\end{longtable}
\end{footnotesize}
\setlength\LTleft{0cm}

\vspace{0.3cm}

\par We denote by $c_{\mathscr{S}}$ the number of $G$-classes of subgroups $M \in \mathscr{S}$ with order divisible by $r$ and $\operatorname{soc}(M)$ not an alternating group.

\begin{corollary}
We have $c_\mathscr{S} \le C_\mathscr{S}$, where $C_\mathscr{S}$ is given in Table 8 for each $G$.

\end{corollary}

\vspace{-0.0cm}

\begin{table}[h]
\begin{center}
\def\arraystretch{1.7}
\caption {Upper bounds for $c_\mathscr{S}$}
\begin{tabular}{ c  c  }
$G$  & $C_\mathscr{S} $   \\ \hline
$PSL_n(q), PSp_n(q), P\Omega_n^-(q), n \ge 7 $ &  $(n^2+\frac{21}{4}n-1)e_G$  \\
$P\Omega^+_n(q), n \ge 10$ &  $(\frac{1}{4}n+9)e_G$  \\
$P\Omega_n(q), n \ge 9, nq$ odd & $(n^2+6n+4)e_G$ \\ 
$PSU_n(q), n \ge 7$ & $3e_G$  \\ \hline
\end{tabular}
\end{center}
\end{table}

\vspace{-0.3cm}

\begin{proof}
\par Let $M \in \mathscr{S}$ with $\operatorname{soc}(M)$ not an alternating group. If $\operatorname{soc}(M) \notin \operatorname{Lie}(p')$, the number $c$ of $GL_n(q)$-conjugacy classes of such $M$ is given in \cite[Example~2.7]{GPPS}. By \cite[Corollary~2.10.4]{KlLi}, the number of $G$-classes is therefore bounded above by $ce_G$. If $\operatorname{soc}(M) \in \operatorname{Lie}(p')$, the number of $G$-classes of $M$ is bounded above by $C_{M}$ in Proposition 4.2. 
\par Suppose $G=PSL_n(q)$. In this case $c \le 6$ by \cite{GPPS}, and so the contribution to $c_\mathscr{S}$ from $M \in \mathscr{S}, \operatorname{soc}(M) \notin \operatorname{Lie}(p')$ is at most $6e_G=6(n, q-1)$. The contribution to $c_{\mathscr{S}}$ from $M \in \mathscr{S}$ with $\operatorname{soc}(M) \in \operatorname{Lie}(p')$ is as follows using Proposition 4.2: 
\begin{itemize}
\item For $\operatorname{soc}(M)=PSL_d(s), d \ge 3$ we have $s \le n$. For each $s$ there exists at most one $d$ such that $n=\frac{s^d-1}{s-1}-1$ as in Table 7. Therefore, by Proposition 4.2, the contribution is less than $\sum_{s=2}^{n}(s-1)e_G=\frac{1}{2}n(n-1)e_G$;
\item For $\operatorname{soc}(M)=PSp_{2d}(s), d \ge 2$ we have $s \le n$ and $s \ne 3$, and since for each $s$ there exists at most one $d$ as above, the contribution is less than $4(n-2)e_G$;
\item For $\operatorname{soc}(M)=PSU_d(s)$ we also have $s \le n$, and, similar to above, the contribution is less than $\sum_{s=2}^{n}(s+1)e_G=\frac{1}{2}(n-1)(n+4)e_G$;
\item For $\operatorname{soc}(M)=PSL_2(s)$ we have $s=n, n+1$ or $2n+1$, contributing $e_G, \frac{1}{4}ne_G, 2e_G$ respectively.
\end{itemize}
\noindent Summing these contributions yields $C_\mathscr{S} = (n^2+\frac{21}{4}n-1)e_G$. The remaining bounds are found in a similar manner.
\end{proof}

\vspace{-1.6cm}

\section*{\center{5. Involutions in maximal subgroups}}
\stepcounter{section}
Let $G$ be a finite simple classical group described in Theorem 2. Recall our definition of $r$ as a primitive prime divisor of $q^e-1$ for $e$ listed in Table 1.

\par In this section we deduce upper bounds for the number of involutions in maximal subgroups of $G$ with order divisible by $r$.

\begin{proposition}
Let $M <_{\operatorname{max}} G$ with $r \mid |M|$. 
\vspace{-4mm}
\begin{enumerate}[label=\roman*)]
  \item If $M \notin \mathscr{S}$, then $i_2(M) \le I_2(M)$, where $I_2(M)$ is listed in Table 6;
  \item If $M \in \mathscr{S}$ and $\operatorname{soc}(M)$ is not an alternating group then $i_2(M)< q^{2n+4}$;
  \item If $M \in \mathscr{S}$ and $\operatorname{soc}(M) = A_{n'}$ for some $n' \ge 9$, then $n'=n+1$ or $n+2$ and $i_2(M)<(n+2)!$.
\end{enumerate}

\end{proposition}

\begin{proof}

\par We prove this case by case for each Aschbacher class. Subgroups $M <_{\operatorname{max}} G$ with $M \notin \mathscr{S}$  and $r \mid |M|$ are listed in Table 6. 

\par First suppose $M \in \mathscr{C}_1$. Then $M$ is of type $Cl_{n-k}(q) \times Cl_{k}(q)$ for some classical group $Cl_k(q)$ of the same type as $G$ and $k=1$ or $2$. We can bound the number of involutions in each factor using Proposition 2.3. If $G=P\Omega^+_n(q)$ and $M$ is of type $O_{n-2}^-(q) \times O_2^-(q)$ we have $i_2(M) \le i_2(PO_{n-2}^-(q)) i_2(PO_2^-(q))$. By Proposition 2.3, $ i_2(PO_{n-2}^-(q)) \le i_2(\operatorname{Aut}(P\Omega_{n-2}^-(q))) < 2(q+1)q^{\frac{n^2-4n}{4}}$. Using the fact that $O_2^-(q) \cong D_{2(q+1)}$ gives $i_2(PO_2^-(q)) \le q+1$, yielding the result. The other values for $I_2(M)$ are calculated in a similar fashion.

\par Next suppose $M \in \mathscr{C}_2, \mathscr{C}_6$ or $\mathscr{S}$. In most cases we bound $i_2(M)$ by $|M|$. For $M \in \mathscr{S}$ with $\operatorname{soc}(M) \ne A_{n'}$ we have $|M|<q^{3n\delta}$ by \cite[Theorem~4.1]{Li}. In fact, using the fact that $r \mid |M|$, by Theorem 4.2 of the same paper we have $|M|<q^{(2n+4)\delta}$ unless $(G,\operatorname{soc}(M)) =(PSL_{11}(2), M_{24}), (P\Omega^+_8(q), P\Omega_7(q))$. In the first case, $i_2(M_{24})=43263<2^{24}$ by \cite{Atlas}, and in the second case, $i_2(M)<2(q+1)q^{11}<q^{20}$ by Proposition 2.3. If $G$ is unitary then $M$ has socle $PSL_2(2n\pm1)$ or $J_3$ by \cite{GPPS}, and each has $i_2(M)<q^{2n+4}$, as can be seen from their respective character tables. If $\operatorname{soc}(M)=A_{n'}, n' \ge 9$ then by \cite{GPPS} we have $n'=n+1$ or $n+2$. Therefore $i_2(M) \le i_2(S_{n+2})<(n+2)!$.

\par If $M \in \mathscr{C}_3$ then $M$ is of type $Cl_k(q^t).t$ for some classical group $Cl_k(q^t)$. Suppose $G=PSL_n(q)$ and let $M$ be of type $GL_k(q^t).t$. If $k \ge 2$ we have $i_2(M) \le \frac{q^t-1}{q-1} i_2(\operatorname{Aut}(PSL_k(q^t)))$, and applying Proposition 2.3 yields the result. If $k=1$, $M$ is of type $GL_1(q^n).n$ where $n$ is prime, and since we are assuming $n>2$ we have $i_2(M)=0$. Other bounds $I_2(M)$ for $M \in \mathscr{C}_3$ are calculated in a similar fashion.

\par If $M \in \mathscr{C}_8$ then $\operatorname{soc}(M)$ is a finite simple classical group, and so we can apply Proposition 2.3 to obtain the result. 
\end{proof}

\vspace{-1.6cm}
\section*{\center{6. Proof of Theorem 2 for $n$ sufficiently large}}
\stepcounter{section}

\par Let $G$ be a finite simple classical group described in Theorem 2. If $G=P\Omega^\epsilon_n(q)$ recall the definition of $a_\epsilon=|PSO^\epsilon_n(q):P\Omega_n^\epsilon(q)|$. Also, recall that $r$ is a primitive prime divisor of $q^e-1$, where $e$ is listed in Table 1. Let $x \in G$ be an element of order $r$. 

\par In this section, except for a small number of possible exceptions (given in Proposition 6.5) we prove Theorem 2 holds in the following cases:
\vspace{-0.2cm}
\begin{align*}
&PSL_n(q), n \ge 9 \\ &PSp_n(q), n \ge 12 \\ &P\Omega^\epsilon_n(q) \  (n \ \text{even}), n \ge 14 \\ &P\Omega_n(q) \ (nq \ \text{odd}), n \ge 13 \\ &PSU_n(q), n \ge 8. 
\end{align*} 

\vspace{-0.33cm}

\noindent We first require a result on the number of conjugates of a maximal subgroup containing $x$.

\begin{lemma}
If $x$ lies in two conjugate maximal subgroups of $G$, say $M$ and $M^g$ for some $g \in G$, then $mg \in N_{G}(\langle x \rangle)$ for some $m \in M$.
\end{lemma}

\begin{proof}
\par We first note that Sylow $r$-subgroups of $G$ are cyclic (this follows from \cite[Theorem 4.10.2]{GoLySo}).
\par Suppose $x \in M \cap M^g$ where $M$ is a maximal subgroup of $G$ and $g \in G$, so $x, x^{g^{-1}} \in M$. We have $\langle x \rangle, \langle x^{g^{-1}} \rangle$ contained in Sylow $r$-subgroups of $M$, conjugate by some $m \in M$. As the Sylow $r$-subgroups are cyclic, $\langle x \rangle$ and $\langle x^{g^{-1}} \rangle$ are also conjugate by $m$. This implies that $mg \in N_{G}(\langle x \rangle)$.
\end{proof}

\begin{corollary}
If $x$ lies in a maximal subgroup $M$ of $G$, the number of $G$-conjugates of $M$ containing $x$ is $\frac{|N_G(\langle x \rangle )|}{|N_M(\langle x \rangle )|}$.
\end{corollary}

\begin{proof}
\par Let $C=\{ M^g: g \in G, x \in M^g \}$. Then $N_G(\langle x \rangle )$ acts on $C$ by conjugation, and Lemma 6.1 implies this action is transitive. The stabilizer of $M$ is $N_G(\langle x \rangle) \cap N_G(M) = N_M(\langle x \rangle)$ since $M$ is maximal. Therefore $|C| = \frac{|N_G(\langle x \rangle)|}{|N_M(\langle x \rangle)|}$, completing the proof.
\end{proof}

\begin{lemma}
Let $T \le G$ be a maximal torus containing $x$. Then $C_G(x)=T$.
\end{lemma}

\begin{proof}
Let $H$ be a simple algebraic group over $\overline{\mathbb{F}_q}$ such that $(H^F)'=G$ for a Steinberg morphism $F$. Also let  $\mathbb{F}_{q^e}^{\times}=\langle \lambda \rangle$. If $\mu \in \mathbb{F}_{q^e}$ is an eigenvalue of $x$, then its Galois-conjugates $\mu^{q^\delta}, \mu^{q^{2\delta}}, \dots$ are also eigenvalues of $x$, where $\delta=2$ if $G$ is unitary and $\delta=1$ otherwise. Therefore, as $x$ is semisimple, $x$ is conjugate in $H$ to $\operatorname{diag}(\lambda^a, \lambda^{aq^\delta}, \lambda^{aq^{2\delta}}, \dots, \lambda^{aq^{e-\delta}}, I_k)$, where $a$ is such that $x$ has order $r$, and $k \le 2$ with equality if and only if $G=P\Omega^+_n(q)$. We find that $C_H(x)$ is contained in a unique maximal torus containing $x$: if $k \le 1$ then $x$ has distinct eigenvalues and the statement is clear, otherwise $k=2, G=P\Omega^+_n(q)$, and the $1$-space is a 2-dimensional torus of type $O_2^+(q)$. Therefore, taking fixed points yields the result. 
\end{proof}

\begin{proposition}
Let $M <_{\operatorname{max}}G$ with $x \in M$, and assume $M \notin \mathscr{S}$. Then $|N_M(\langle x \rangle)| \ge N_M$, where $N_M$ is listed in Table 6. The values of $|N_G(\langle x \rangle)|$ are listed in Table 9. 
\end{proposition}

\renewcommand{\arraystretch}{1.7}
\setlength\LTleft{4.85cm}

\begin{longtable}{c c}
\caption {Orders of normalizers of $x$} \\
$G$  & $ |N_G(\langle x \rangle)|$ \\   \hline
$PSL_n(q)$ & $\frac{n(q^n-1)}{(q-1)(n, q-1)}$ \\
$PSp_n(q)$ & $ \frac{n(q^\frac{n}{2}+1)}{(2, q-1)}$ \\
$P\Omega^+_n(q)$ & $\frac{(n-2)(q^{\frac{n}{2}-1}+1)(q+1)}{a_+(2, q-1)^2}$ \\
$P\Omega_n^-(q)$ & $\frac{n(q^\frac{n}{2}+1)}{a_-(2, q-1)}$ \\
$P\Omega_n(q)$ & $\frac{(n-1)(q^{\frac{n-1}{2}}+1)}{2}$ \\
$PSU_n(q), n$ odd & $\frac{n(q^n+1)}{(q+1)(n, q+1)}$ \\
$PSU_n(q), n$ even & $\frac{(n-1)(q^{n-1}+1)}{(n, q+1)}$ \\ \hline
\end{longtable}
\setlength\LTleft{0cm}

\begin{proof}
We first consider $N_G(\langle x \rangle)$. By Lemma 6.3 there is a unique maximal torus $T$ of $G$ containing $x$. We have $N_G(\langle x \rangle) \le N_G(T)$ by Lemma 6.3, and conversely the normalizer of $T$ will normalize the unique Sylow $r$-subgroup of $T$ containing $x$, yielding $N_G(\langle x \rangle)=N_G(T)$. By \cite[\S~25]{MaTe}, $T$ corresponds to an element $w$ in the Weyl group $W$ of $G$, and  Proposition 25.3 in the same section gives $|C_W(w)|=|N_G(T):T|$. The values of $|T|$ and $|N_G(T):T|$ can be computed from \cite{ButGre}. Hence $|N_G(\langle x \rangle)|$ as in Table 9. 

\par We now consider $N_M(\langle x \rangle)$, and proceed case by case for each Aschbacher class of subgroups $M$ containing $x$ with $r \mid |M|$. Such classes are listed in Table 6.

\par Consider $M \in \mathscr{C}_1$ of type $Cl_{n-k}(q) \times Cl_{k}(q)$ for some classical group $Cl_m(q)$ of the same type as $G$ with $k \le 2$. Then $N_M(\langle x \rangle)$ is of the form $ N_{Cl_{n-k}(q)}(\langle x \rangle) \times Cl_{k}(q)$. If the maximal torus in $Cl_{n-k}(q)$ containing $x$ is $T'$, then, by the same reasoning as for $N_G(\langle x \rangle)$ in first paragraph, the normalizer of $x$ in $M$ is of the form $N_{Cl_{n-k}(q)}(T') \times Cl_k(q)$, and using \cite{ButGre} yields the result.

\par For $M \in \mathscr{C}_2$ or $\mathscr{C}_6$ we use the obvious bound $|N_M(\langle x \rangle)| \ge r$.

\par Consider $M \in \mathscr{C}_3$ of the form $M_0.t$ where $M_0=Cl_k(q^t)$, a classical group of the same type as $G$. Then $T \le M_0$, and $N_{M_0}(\langle x \rangle)=N_{M_0}(T)$ as above. It remains to consider $\frac{M}{M_0}$. By Lemma 6.3, $C_G(x)=T$, and so there is a unique Sylow $r$-subgroup $P$ of $M$ containing $x$. Therefore $N_M(\langle x \rangle)=N_M(P)$. By the Frattini argument, $M=M_0 N_M(P)$, and so
\begin{align*} 
 \frac{N_M(\langle x \rangle)}{N_{M_0}(\langle x \rangle)} = \frac{M_0N_M(\langle x \rangle)}{M_0} = \frac{M_0N_M(P)}{M_0} = \frac{M}{M_0}=C_t.
\end{align*}

\noindent Hence $N_M(\langle x \rangle)=N_{M_0}(\langle x \rangle).t$, showing that for $M \in \mathscr{C}_3$ we have $N_G(\langle x \rangle)=N_M(\langle x \rangle)$. For the remaining cases $M \in \mathscr{C}_3$ a similar method yields the result.

\par For $M \in \mathscr{C}_8$, we use the same method as we did for $N_G(\langle x \rangle)$ in the first paragraph to obtain the lower bound $N_M$.  \end{proof}

\begin{proposition}
Apart from 4 possible exceptions, Theorem 2 holds for $G$ in the following cases: 
\vspace{-0.4cm}
\begin{align*}
&PSL_n(q), n \ge 9; \\ &PSp_n(q), n \ge 12; \\ &P\Omega^\epsilon_n(q) \  (n \ \text{even}), n \ge 14; \\ &P\Omega_n(q) \ (nq \ \text{odd}), n \ge 13; \\ &PSU_n(q), n \ge 8. 
\end{align*} 
\noindent The possible exceptions are $G=PSp_{12}(2), P\Omega^+_{14}(2), P\Omega^+_{16}(2), P\Omega^+_{18}(2)$.
\end{proposition}

\begin{proof}
\par To prove $G$ is $(2,r)$-generated, by $(1.2)$  it suffices to prove $\sum_{x \in M <_{\text{max}}G} \frac{i_2(M)}{i_2(G)}<1$. 
\par Let $\mu_i$ be a set of $G$-conjugacy class representatives of maximal subgroups $M \in \mathscr{C}_i$ with order divisible by $r$, and $\mu_0$ be a set of $G$-conjugacy class representatives for $M \in \mathscr{S}$ with order divisible by $r$. By Corollary 6.2 we have
\begin{align*}
\sum_{x \in M <_{\text{max}}G} \frac{i_2(M)}{i_2(G)}=\Sigma_1 + \Sigma_2+ \dots + \Sigma_8 + \Sigma_0
\end{align*} 
\noindent where
\begin{align*}
 \Sigma_i = \sum_{ M \in \mu_i } \frac{|N_G(\langle x \rangle)|}{|N_{M}(\langle x \rangle)|} \frac{i_2(M)}{i_2(G)}.
\end{align*}
\noindent The values $|\mu_i|$ and maximal subgroups $M$ contributing to $\Sigma_i, 1 \le i \le 8$ are found in Table 6 by Proposition 4.1. Upper bounds for the number of conjugacy class representatives $M \in \mu_0$ with $\operatorname{soc}(M) \notin \{ A_{n+1}, A_{n+2} \} $ are found in Table 8 by Corollary 4.3. For $M \in \mu_0$ with $\operatorname{soc}(M) \in \{ A_{n+1}, A_{n+2} \},$ $\operatorname{soc}(M)$ has a unique irreducible $n$-dimensional representation over any field (where the characteristic $p \mid n+2$ in the $n+2$ case), preserving an orthogonal (or symplectic in characteristic 2) form (see \cite[\S~5]{KlLi}). Hence the number of such representations is 0 if $G=PSL^\epsilon_n(q),$ and is at most $e_G$ otherwise. 
\par By Proposition 6.4, if $M \notin \mathscr{S}$ then $\frac{|N_G(\langle x \rangle)|}{|N_{M}(\langle x \rangle)|} \le  \frac{|N_G(\langle x \rangle)|}{N_M}$ where $N_M$ is found in Table 6 and $|N_G(\langle x \rangle)|$ is found in Table 9. If $M \in \mathscr{S}$ we use $|N_M(\langle x \rangle)| \ge r$ if $\operatorname{soc}(M) \notin \{A_{n+1}, A_{n+2} \}$, and $|N_M(\langle x \rangle)| \ge \frac{1}{2}r(r-1)$ otherwise. The ratio $\frac{i_2(M)}{i_2(G)}$ is bounded by $\frac{I_2(M)}{I_2(G)}$  by Proposition 3.1 and Proposition 5.1. This leads to an upper bound for each $\Sigma_i$ which can be manipulated into a decreasing function in $n$ and $q$.

\par As an illustration, consider $G=P\Omega^-_n(q)$. By Table 6, 
\begin{align*} 
\sum_{x \in M <_{\text{max}}G} \frac{i_2(M)}{i_2(G)}= \Sigma_3+\Sigma_0.
\end{align*}
 \vspace{-0.3cm}
\par Consider first $\Sigma_3$. In $\mathscr{C}_3$ there are fewer than $\frac{n}{2}$ classes of type $O^-_k(q^t).t$ and there is 1 class of type $GU_\frac{n}{2}(q).2$. By Corollary 6.2 and Table 6, if $M \in \mu_3$ then $M$ is the unique conjugate containing $x$ since $|N_G(\langle x \rangle)|=|N_M(\langle x \rangle)|$. By Table 6, for $M$ of type $O^-_k(q^t).t$ we have $i_2(M) < 2(q^t+1)q^{\frac{n^2}{4t}-t} \le 2(q^2+1)q^{\frac{n^2}{8}-2}$. For $M$ of type $GU_\frac{n}{2}(q).2$ we have $i_2(M) < \frac{2}{z_-}(q+1)^2q^{\frac{n^2}{8}+\frac{n}{4}-2}$. By Table 4, we have $i_2(G) \ge \frac{1}{8}q^{\frac{n^2}{4}-1}$. Therefore
\begin{align*}
\Sigma_{3} &=  \sum_{x \in M \in \mathscr{C}_3}   \frac{i_2(M)}{i_2(G)} \\
&\le \sum_{ M \in \mu_3 } \frac{|N_G(\langle x \rangle)|}{|N_{M}(\langle x \rangle)|} \frac{i_2(M)}{i_2(G)} \\
&\le   \frac{\frac{n}{2}.2(q^2+1)q^{\frac{n^2}{8}-2}+ \frac{2}{z_-}(q+1)^2q^{\frac{n^2}{8}+\frac{n}{4}-2}}{\frac{1}{8}q^{\frac{n^2}{4}-1}} \\
&= \frac{2^3n(q^2+1)}{q^{\frac{n^2}{8}+1}}+\frac{2^4(q+1)^2}{q^{\frac{n^2}{8}-\frac{n}{4}+1}}. \tag{1}
\end{align*}

\par We now consider $\Sigma_0$. By Corollary 4.3 there are at most $(n^2+\frac{21}{4}n-1)e_G$ classes of subgroups $M \in \mu_0$ such that $\operatorname{soc}(M) \notin \{ A_{n+1}, A_{n+2} \}$. As shown above, there are at most $e_G$ classes of subgroups $M \in \mu_0$ such that $\operatorname{soc}(M) \in  \{ A_{n+1}, A_{n+2} \}$. By Corollary 6.2, for $M$ of the first type there are at most $ \frac{n(q^\frac{n}{2}+1)}{ra_-(2,q-1)}$ $G$-conjugates of $M$ containing $x$, and for $M$ of the second type there are at most $\frac{n(q^\frac{n}{2}+1)}{\frac{1}{2}r(r-1)a_-(2,q-1)}$. By Proposition 5.1 we have $i_2(M)<q^{2n+4}$ for $M$ of the first type, and $i_2(M)< (n+2)!$ for the second type. Therefore
\begin{align*}
\Sigma_0 &= \sum_{x \in M \in \mathscr{S}}  \frac{i_2(M)}{i_2(G)} \\
&\le \sum_{ M \in \mu_0 } \frac{|N_G(\langle x \rangle)|}{|N_{M}(\langle x \rangle)|} \frac{i_2(M)}{i_2(G)} \\
&\le \frac{e_Gn(n^2+\frac{21}{4}n-1)(q^\frac{n}{2}+1)q^{2n+4}}{ra_-(2,q-1)\frac{1}{8}q^{\frac{n^2}{4}-1}} + \frac{e_Gn(q^\frac{n}{2}+1)(n+2)!}{\frac{1}{2}r(r-1)a_-(2,q-1)\frac{1}{8}q^{\frac{n^2}{4}-1}} \\
&\le \frac{2^3(2,q-1)(n^2+\frac{21}{4}n-1)(q^\frac{n}{2}+1)}{q^{\frac{n^2}{4}-{2n}-5}} + \frac{2^4(2,q-1)(n+2)!(q^\frac{n}{2}+1)}{nq^{\frac{n^2}{4}-1}}. \tag{2}
\end{align*}

\noindent Using (1), (2) it can then be verified that for $n \ge 16$ and all $q$, or for $n \ge 14$ and $q \ge 3$, we have $\Sigma_3+\Sigma_0<1$. For $n=14, q=2$, we find $r=43$, and so by Table 7 and \cite{GPPS} there are no subgroups $M \in \mathscr{S}$ with order divisible by $r$. Therefore using (1) we find $\Sigma_3 < 1$, and this proves the result for $G=P\Omega^-_n(q), n \ge 14$. Similar arguments deal with all other possibilities for $G$.
\end{proof}

\vspace{-1.6cm}
\section*{\center{7. Proof of Theorem 2 for small $n$}}
\stepcounter{section}

\par We now consider $G$ with $n$ smaller than in Proposition 6.5. To prove Theorem 2 it suffices to prove the following groups are $(2,r)$-generated: 
\vspace{-0.21cm}
\begin{align*}
&\hspace{0.95cm} PSL_8(q); \\
&\hspace{0.95cm} PSp_n(q), n=8, 10 \ \text{and} \ (n,q)=(12, 2) ; \\
&\hspace{0.95cm} P\Omega^\epsilon_n(q), n=8 \ (q \ne 2 \ \text{for} \ \epsilon=+), 10, 12 \ \text{and} \ (n,q, \epsilon)=(14,2, +), (16,2,+), (18,2,+);   \\
&\hspace{0.95cm} P\Omega_n(q) \ (q \ \text{odd}), n=9, 11.
\end{align*}

\vspace{-2.1cm} \hspace{-0.8cm} $(\dagger)$

\vspace{0.99cm}

\begin{proposition}
If $G$ is a group listed in $(\dagger)$ and $M$ is a maximal subgroup of $G$ with order divisible by $r$, then $M$ is conjugate to a group listed in Table 6 ($M \in \mathscr{C}_i$) or a group with socle listed in Table 10 ($M \in \mathscr{S}$). In Table 10, upper bounds $C_{M}$ are given for the number of $G$-classes of subgroups $M$ with the given socle.
\end{proposition}

\setlength\LTleft{2.52cm}
\renewcommand{\arraystretch}{1.3}
\footnotesize
\begin{longtable}{c c  c  c }
\caption{Maximal subgroups $M \in \mathscr{S}$ with $r \mid |M|$ for $G$ in $(\dagger)$} \\
$G$  & $\operatorname{soc}(M)$ & Conditions & $C_{M}$\\ \hline
$PSp_8(q)$ &   $PSL_2(17)$ & $q=p$ or $p^2$, $q \ge 9$ or $q=2$, $r=17$  & 2 \\ \hline
$PSp_{10}(q)$ &   $PSL_2(11)$ & $q=p$ or $p^2$, $q$ odd, $r=11$ & 6 \\ 
 & $PSU_5(2)$ & $q=p$ odd, $r=11$ & 2 \\\hline
$PSp_{12}(2)$ & $PSL_2(25)$ & & 1 \\ 
& $A_{14}$ &  & 1 \\ \hline
$P\Omega^+_8(q)$  & $P\Omega_7(q) $ & $q$ odd  & $4$ \\
& $PSp_6(q) $ &  $q$ even & 2 \\
& $PSU_3(q) $ & $q \equiv 2 \mod 3$ & $(2,q-1)^2$ \\
& $ P\Omega^+_8(2)$ & $q=p$ odd, $r=7$ & 4 \\
& $ Sz(8) $ & $q=5$ & 8 \\
& $A_{10}$ & $q=5$ & 12 \\ \hline
$P\Omega^+_{12}(q)$ & $PSL_2(11)$ & $q=p \ge 19$, $r=11$ & 8 \\
& $M_{12}$ & $q=p \ge 5$, $r=11$ & 8 \\
& $A_{13}$ & $q=p$ odd, $r=11$ & 4 \\ \hline
$P\Omega^+_{14}(2)$ & $ PSL_2(13)$ & & $2e_G$ \\
& $ G_2(3)$ & & $e_G$\\
&  $A_{16}$ & & $e_G$\\ \hline
$P\Omega_{10}^-(q)$ &  $PSL_2(11)$ & $q=p \ge 11$, $r=11$ & $(q+1, 4)$\\
& $ A_{11}$ & $q \ne 2$, $r=11$ & $(q+1, 4)$ \\
& $ A_{12}$ & $q=2$ & $1$ \\ \hline
$P\Omega_{12}^-(q)$ &   $PSL_2(13)$ & $q=p$ or $p^3$, $q \ge 8$, $r=13$ & 6 \\
& $PSL_3(3)$ & $q=p$, $r=13$ & $2(q+1, 2)$\\
& $ A_{13} $ & $q \ne 7$, $r=13$ & $(q+1, 2)$  \\ \hline
$P\Omega_9(q)$ & $PSL_2(17)$ & $q=p$ or $p^2$, $r=17$ & $2$ \\ \hline
$P\Omega_{11}(q)$ & $ A_{12}$ & $q=p$, $r=11$ & 2 \\
\hline
\end{longtable}

\normalsize
\begin{proof} 
For $n \le 12$, this is proved using \cite{BHRD}. For the groups $G=P\Omega^+_{14}(2), P\Omega^+_{16}(2), P\Omega^+_{18}(2)$ with $M \in \mathscr{S}$, a list of possibilities for $\operatorname{soc}(M)$ and the number of $GL_n(q)$-classes $c$ for each $\operatorname{soc}(M)$ is obtained from \cite{GPPS}. By \cite[Corollary~2.10.4]{KlLi} we can bound the number of $G$-classes of $M$ by $ce_G$. \end{proof}

\begin{proposition}
Theorem 2 holds for $G$ listed in $(\dagger)$.
\end{proposition}

\begin{proof}
We prove the remaining cases of $G$ are all $(2,r)$-generated in the usual way, proceeding case by case for $n$. We use similar bounds for $\Sigma_i, 1 \le i \le 8$, to those found in the proof of Proposition 6.5, though since $n$ is fixed in each case we are able to improve the bounds on $\Sigma_2$ and $\Sigma_3$ in the following way:
\begin{itemize}

\item Consider $M \in \mathscr{C}_2$ of the form $O_1(q) \operatorname{wr} S_n$. Then we have $i_2(M) \le 2^{n-1}(i_2(S_n)+1) = 2^{n-1}\sum_{k=0}^{\lfloor \frac{n}{2} \rfloor} \frac{n!}{2^k k! (n-2k)!}$.
\item Let $M \in \mathscr{C}_3$ be of the form $Cl_k(q^t).t$ for a classical group $Cl_k(q^t)$ of the same type as $G$. Instead of bounding the number of prime divisors of $n$ (and hence the number of classes) by $dn$ for some $d \le 1$, we instead calculate the exact number of prime divisors and bound the number of involutions contained in each class separately using Table 6. 
\end{itemize}
\par We can improve the bound on $\Sigma_0$ more significantly: Table 10 lists possible $\operatorname{soc}(M)$ for subgroups $M$ contributing to $\Sigma_0$. If $\operatorname{soc}(M) \notin \{A_{n+1}, A_{n+2} \}$ then we bound $i_2(M)$ using either $|M|$ or Proposition 2.3 if applicable, rather than $q^{2n+4}$ from \cite{Li}. If $\operatorname{soc}(M) =A_{n'}$ for $n'=n+1$ or $n+2$, then $i_2(M) \le i_2(S_{n'}) \le \sum_{k=0}^{\lfloor \frac{n'}{2} \rfloor} \frac{n'!}{2^k k!(n'-2k)!}$.
\par We note that for each $n$ it may not be possible to prove $Q_2(G,x)<1$ for all $q$ using the lower bound $i_2(G) \ge I_2(G)$. However, for specific $q$ we can prove $Q_2(G,x)<1$ by instead using the lower bound for $i_2(G)$ listed in Table 5.
\par As an illustration, consider $G=P\Omega^+_{12}(q)$. By Table 6 and Table 10, 
\begin{align*}
\sum_{x \in M <_{\text{max}}G} \frac{i_2(M)}{i_2(G)} =\Sigma_1 + \Sigma_2+\Sigma_3+\Sigma_0.
\end{align*}
\vspace{-0.4cm}
\par We first consider $\Sigma_1$. Let $\mu_1$ be a set of conjugacy class representatives of subgroups $M \in \mathscr{C}_1$ such that $r \mid |M|$. There exists a unique $M \in \mu_1$ of type $O_{10}^-(q) \times O_2^-(q)$ and either two representatives of type $O_{11}(q) \times O_1(q) \ (q \ \text{odd})$ or one of type $O_{11}(q) \ (q \ \text{even})$. Using Tables 4, 6 and 9 we bound $c_M \frac{ |N_G(\langle x \rangle)|}{|N_M(\langle x \rangle)|} \frac{i_2(M)}{i_2(G)}$ by $c_M \frac{|N_G(\langle x \rangle)|}{N_M} \frac{I_2(M)}{I_2(G)}$, and we find that the subgroups giving the largest contribution to $\Sigma_1$ occur when $q$ is odd. This leads to 

\vspace{-0.5cm}
\begin{align*}
\Sigma_1 &= \sum_{M \in \mu_1} \frac{|N_G(\langle x \rangle)|}{|N_M(\langle x \rangle)|} \frac{i_2(M)}{i_2(G)} \\
&\le \sum_{\substack{M \in \mu_1 \ \text{of type}\\  O_{10}^-(q) \times O_2^-(q) }} \frac{|N_G(\langle x \rangle)|}{|N_M(\langle x \rangle)|} \frac{i_2(M)}{i_2(G)} + \sum_{\substack{M \in \mu_1 \ \text{of type} \\  O_{11}(q) \times O_1(q) }}\frac{|N_G(\langle x \rangle)|}{|N_M(\langle x \rangle)|} \frac{i_2(M)}{i_2(G)} \\
&\le \frac{2(q+1)^2q^{24}}{\frac{1}{8}q^{35}} + \frac{2(q+1)}{(2, q-1)^2} \times \frac{\frac{4}{z_+}(q+1)q^{29}}{\frac{1}{8}q^{35}} \\
&= \frac{2^4(q+1)^2}{q^{11}}+\frac{2^6(q+1)^2}{(2,q-1)^3q^6}.
\end{align*}
\vspace{-0.5cm}
\par We now consider $\Sigma_2$. Let $\mu_2$ be a set of conjugacy class representatives for $M \in \mathscr{C}_2$ such that $r \mid |M|$. By Table 6, $| \mu_2| \le 4$, and for each $M \in \mu_2$ we have $i_2(M) \le  2^{11}\sum_{k=0}^{6} \frac{12!}{2^k k! (12-2k)!}=2^{14}.17519$ from the above discussion. By Table 6 we require $q=p \ne 2$ for such subgroups to exist. This yields
\begin{align*}
\Sigma_2 &=\sum_{M \in \mu_2} \frac{|N_G(\langle x \rangle)|}{|N_M(\langle x \rangle)|} \frac{i_2(M)}{i_2(G)} \\
&\le  4 \times \frac{10(q^5+1)(q+1)}{4a_+r} \times \frac{2^{11}\sum_{k=0}^{6} \frac{12!}{2^k k! (12-2k)!}}{\frac{1}{8}q^{35}} \\
&\le \frac{2^{16}.17519(q+1)(q^5+1)}{q^{35}}.
\end{align*}
\vspace{-0.5cm}

\par We now consider $\Sigma_3$. Let $\mu_3$ be a set of conjugacy class representatives for $M \in \mathscr{C}_3$. By Table 6 there are 2 such classes of the form $GU_6(q).2$, and for each we have $i_2(M)<\frac{2}{z_+}(q+1)^2q^{19}$. Each class has a unique conjugate containing $x$ by Corollary 6.2. Therefore
\begin{align*}
\Sigma_3 &= \sum_{M \in \mu_3} \frac{|N_G(\langle x \rangle)|}{|N_M(\langle x \rangle)|}  \frac{i_2(M)}{i_2(G)} \\
&\le 2 \times \frac{2(q+1)^2q^{19}}{\frac{1}{8}q^{35}} \\
&= \frac{2^5(q+1)^2}{q^{16}}.
\end{align*}
\vspace{-0.5cm}
\par We now consider $\Sigma_0$. By Table 10, we can assume $r=11$. Let $\mu_0$ be a set of conjugacy class representatives for $M \in \mathscr{S}$ with $r \mid |M|$. By Table 10 there are at most $8$ classes with socle $PSL_2(11)$, $8$ classes with socle $M_{12}$, and $4$ classes with socle $A_{13}$. By Corollary 6.2 and Table 9, for $M$ of the first or second type there are at most $\frac{10(q^5+1)(q+1)}{11a_+(2, q-1)^2}$ $G$-conjugates of $M$ also containing $x$, and for $M$ of alternating type there are at most $\frac{10(q^5+1)(q+1)}{\frac{1}{2}11(11-1)a_+(2, q-1)^2}$ conjugates containing $x$. Using \cite{BHRD} and \cite{Atlas} we calculate $i_2(M) \le 55, 190080, 272415$ for $\operatorname{soc}(M)=PSL_2(11), M_{12}, A_{13}$ respectively. Therefore, as $q=p$ is odd in all cases, 
\begin{align*}
\Sigma_0 &= \sum_{M \in \mu_0} \frac{|N_G(\langle x \rangle)|}{|N_M(\langle x \rangle)|}  \frac{i_2(M)}{i_2(G)} \\
&=  8 \times \frac{10(q^5+1)(q+1)}{ra_+(2, q-1)^2} \times \frac{55}{\frac{1}{8}q^{35}} + 8 \times \frac{10(q^5+1)(q+1)}{ra_+(2, q-1)^2} \times \frac{190080}{\frac{1}{8}q^{35}}  \\ & \hspace{3cm} + 4 \times \frac{10(q^5+1)(q+1)}{\frac{1}{2}r(r-1)a_+(2, q-1)^2} \times \frac{272415}{\frac{1}{8}q^{35}} \\
&= \frac{2^3.5.11.3593(q^5+1)(q+1)}{q^{35}}.
\end{align*}
\vspace{-0.5cm}
\par We see that for $q \ge 3$ we have $Q_2(G,x) \le \Sigma_1+\Sigma_2+\Sigma_3+\Sigma_0<1$. It therefore suffices to prove $\Sigma<1$ for $q=2$. Computing similar bounds for $\Sigma_i$ using the lower bound for $i_2(G)$ found in Table 5 instead of $I_2(G)$ from Table 4 yields the result.
\par Calculations for the remaining $G$ are similar.
\end{proof}

\end{document}